\documentclass[12pt,reqno]{amsart}
\usepackage{amssymb,amsthm,amsmath,amstext,amsxtra}
\usepackage{mathrsfs,bm}
\usepackage{fullpage}
\usepackage[all]{xy}
\usepackage{booktabs}
\usepackage{mathtools}
\usepackage{hyperref}
\hypersetup{colorlinks=true,urlcolor=blue,citecolor=blue,linkcolor=blue}
\usepackage{enumerate}
\usepackage{ stmaryrd }
\usepackage{enumitem}
\usepackage{comment}
\usepackage{multirow}
\usepackage{colonequals}
\usepackage{tikz}
\usepackage{marginnote}
\usepackage[export]{adjustbox}

\usepackage{tikz}
\usepackage{tikz-cd}
\usepackage{xcolor}

\numberwithin{equation}{subsection}

\newtheorem{theorem}{Theorem}[section]
\newtheorem{lemma}[theorem]{Lemma}

\newtheorem{conjecture}[theorem]{Conjecture}
\newtheorem{remark}[theorem]{Remark}

\newtheorem{proposition}[theorem]{Proposition}

\newtheorem{corollary}[theorem]{Corollary}

\begin{document}

\title{Double Zeta Values and Picard-Fuchs equation}

\author{Wenzhe Yang}
\address{SITP Stanford University, CA, 94305}
\email{yangwz@stanford.edu}

\begin{abstract}
In this paper we will study the double zeta values $\zeta(k,m)$ using Picard-Fuchs equation. We will give a very efficient method to evaluate $\zeta(k,1)$ (resp. $\zeta(k,2)$) in terms of the products of zeta values $\zeta(2),\zeta(3),\cdots$ when $k$ is even (resp. odd), which admits immediate generalization to arbitrary double zeta values. Moreover, this method provides new insights into the nature of double zeta values, which further can be generalized to arbitrary multiple zeta values.
\end{abstract}

\maketitle
\setcounter{tocdepth}{1}
\vspace{-13pt}
\tableofcontents
\vspace{-13pt}

\section{Introduction}
The multiple zeta functions, as generalizations of the Riemann zeta function, are defined by the infinite sum
\begin{equation}
\zeta(s_1,\cdots,s_l)=\sum_{n_1>n_2>\cdots >n_l>0} \frac{1}{n_1^{s_1}\cdots n_l^{s_l}},
\end{equation}
which converge on the region where $\text{Re}(s_1)+\cdots \text{Re}(s_i)>i$ for all $i$. Like the Riemann zeta function, the multiple zeta functions can also be analytically continued to meromorphic functions on $\mathbb{C}^l$. When $s_1,\cdots,s_l$ are positive integers with $s_1 \geq 2$, these infinite sums are called multiple zeta values (MZVs), which are very important objects in number theory. The integer $l$ is called the length of the MZV, while $\sum_{i=1}^l s_i$ is called the weight of the MZV. In this paper, we will focus our attention on the case where $l=2$, i.e. double zeta values (DZVs), and we will use Picard-Fuchs equation to study them. 

We now briefly explain the method for the DZV $\zeta(k,1),~k\geq 2$. By definition, $\zeta(k,1)$ is given by
\begin{equation}
\zeta(k,1)=\sum_{n> m \geq 1} \frac{1}{n^k m}=\sum_{n=1}^\infty \frac{H_{n,1}}{(n+1)^k},
\end{equation}
where the harmonic number $H_{n,1}$ is defined by $\sum_{m=1}^n 1/m$. Next, we construct a power series
\begin{equation}
\Pi_{k,1}=\sum_{n=1}^\infty \frac{(-1)^n H_{n,1}}{(n+1)^k} \phi^{n+1},
\end{equation}
which satisfies the Picard-Fuchs equation
 \begin{equation}
\left( (1+\phi)^2 \vartheta^{k+2}-3(1+\phi)\vartheta^{k+1}+(2+\phi)\vartheta^k \right)\Pi_{k,1}=0,~ \vartheta=\phi \frac{d}{d \phi}.
 \end{equation}
Let $\varphi$ be $1/\phi$, and on the unit disc $|\varphi|<1$, the $k+2$ dimensional solution space of the Picard-Fuchs equation has a canonical basis of the form
\begin{equation}
\begin{aligned}
\varpi^{k,1}_i&=\log^i \varphi,~i=0,1,\cdots,k-1;\\
\varpi^{k,1}_{k}&=\log^k \varphi+\sum_{n=1}^\infty k! \frac{(-1)^n}{n^k}\, \varphi^n;\\
\varpi^{k,1}_{k+1}&=\log^{k+1} \varphi+(k+1)\left( \sum_{n=1}^\infty k! \frac{(-1)^n}{n^k} \varphi^{n} \right) \log \varphi \\
&+(1-k)(k+1)!\sum_{n=1}^\infty \frac{(-1)^n}{n^{k+1}} \varphi^n-(k+1)! \sum_{n=1}^\infty \frac{(-1)^nH_{n,1}}{(n+1)^k} \varphi^{n+1}.
\end{aligned}
\end{equation}
With respect to this canonical basis, there exist $k+2$ complex numbers $\{\tau^{k,1}_i \}_{i=0}^{k+1}$ such that
\begin{equation}
\Pi_{k,1}=\sum_{i=0}^{k+1} \tau^{k,1}_i \,\varpi^{k,1}_i,
\end{equation}
while $\tau^{k,1}_i$ can be computed by Fourier analysis on the unit circle
\begin{equation}
S^1=\{|\phi|=1 \}=\{|\varphi|=1 \}.
\end{equation}
More explicitly, for various $n$, the equation
\begin{equation} \label{eq:introductionlineartaui}
\int_{-1/2}^{1/2} \phi^{-n} \Pi_{k,1} dt=\sum_{i=0}^{k+1} \tau^{k,1}_i \int_{-1/2}^{1/2} \varphi^n \varpi^{k,1}_i dt
\end{equation}
gives us various linear equations about $\tau^{k,1}_i$, solving which yields the values of $\tau^{k,1}_i$. In fact, the integrals
\begin{equation}
\int_{-1/2}^{1/2} \phi^{-n} \Pi_{k,1} dt~\text{and}~ \int_{-1/2}^{1/2} \varphi^n \varpi^{k,1}_i dt
\end{equation}
can be evaluated explicitly, and they lie in the field $\mathbb{Q}(\pi,\zeta(2),\zeta(3),\cdots,\zeta(k+1))$, hence we conclude that $\tau^{k,1}_i$ also lies in this field. When $k=2$, we have
\begin{equation}
\tau^{2,1}_0=-\zeta(3),~\tau^{2,1}_1=0,~\tau^{2,1}_2=0,~\tau^{2,1}_3=\frac{1}{6}.
\end{equation}
We have also computed $\tau^{k,1}_i$ for $k=3,4,\cdots,9$, and our computations have shown that
\begin{equation}\label{eq:introductionrecursive}
\tau^{k,1}_i=-\frac{1}{i}\tau^{k-1,1}_{i-1},~i \geq 1,
\end{equation}
so we conjecture this equation is valid for all $k$. On the other hand, if we let $n=0$ in the formula \ref{eq:introductionlineartaui}, we have 
\begin{equation} \label{eq:introductiontau0}
\tau^{k,1}_0=- \sum_{i=1}^{k+1} \tau^{k,1}_i \int_{-1/2}^{1/2}  \varpi^{k,1}_i dt,
\end{equation}
while the integrals appear in this equation are given by
\begin{equation}
\begin{aligned}
\int_{-1/2}^{1/2}  \varpi^{k,1}_j dt&=\frac{\left(1+(-1)^j \right) (\pi i)^j}{2(1+j)},~j=1,\cdots,k;\\
\int_{-1/2}^{1/2}  \varpi^{k,1}_{k+1} dt&=\frac{\left(1+(-1)^{k+1} \right) (\pi i)^{k+1}}{2(k+2)}+(k+1)!\zeta(k+1).\\
\end{aligned}
\end{equation}
The upshot is that formulas \ref{eq:introductionrecursive} and \ref{eq:introductiontau0} together give us a very efficient method to explicitly compute $\tau^{k,1}_i$ in terms of the zeta values. In particular, we have
\begin{equation}
\tau^{k,1}_{k+1}=(-1)^k/(k+1)!.
\end{equation}
When $k$ is even, $\tau^{k,1}_{k+1}$ is positive, and the equation
\begin{equation} \label{eq:introductionk1minus1}
\Pi_{k,1}(-1)=\sum_{i=0}^{k+1} \tau^{k,1}_i \varpi^{k,1}_i(-1)
\end{equation}
in fact yields an evaluation of $\zeta(k,1)$ in terms of the zeta values $\zeta(2),\cdots,\zeta(k+1)$. For example, when $k=2$, this equation is equivalent to
\begin{equation}
\zeta(2,1)=\zeta(3),
\end{equation}
which is first proved by Euler. Similarly, the equation 
\begin{equation}  \label{eq:introductionk11}
\Pi_{k,1}(1)=\sum_{i=0}^{k+1} \tau^{k,1}_i \varpi^{k,1}_i(1)
\end{equation}
yields an evaluation of 
\begin{equation}
\sum_{n=1}^\infty \frac{(-1)^nH_{n,1}}{(n+1)^k},
\end{equation}
in terms of the zeta values. However when $k$ is odd, the two equations \ref{eq:introductionk1minus1} and \ref{eq:introductionk11} only give us the trivial identity $0=0$. We will use the same method to study $\zeta(k,2)$, and it also gives us an efficient method to evaluate $\zeta(k,2)$ in terms of the zeta values when $k$ is odd. The method in this paper certainly admits generalizations to arbitrary double zeta values, which provides new insights into these interesting numbers. Moreover, this method can also be applied to study MZVs.

The outline of this paper is as follows. Section \ref{sec:introductoryexamplepi} is a short review about how our method works for a toy example. Section \ref{sec:zetak1section} studies the double zeta values $\zeta(k,1)$. Section \ref{sec:doublezetak2} is about $\zeta(k,2)$. In Section \ref{sec:furtherprospects}, we will list several interesting open questions.

\section{A Toy example} \label{sec:introductoryexamplepi}

In this section, we will introduce a toy example that nevertheless illustrates the idea of the method in this paper. The value of the Riemann zeta function $\zeta(s)$ at the integer point $s=2$ is given by
\begin{equation}
\zeta(2)=\sum_{n=1}^\infty \frac{1}{n^2}=\frac{\pi^2}{6}.
\end{equation}
We now construct a power series $\Pi_0$ of the form
\begin{equation}
\Pi_0(\phi)=\sum_{n=1}^\infty \frac{(-1)^n}{n^2} \phi^n,
\end{equation}
which converges on the unit disc $|\phi|<1$. Moreover, $\Pi_0(\phi)$ actually converges absolutely on the unit circle $S^1$ ($|\phi|=1$), thus it defines a continuous function on it, while the value of $\Pi_0$ at $\phi=-1$ is just $\zeta(2)$. It is well-known that the power series $\Pi_0(\phi)$ satisfies a third order Picard-Fuchs equation
\begin{equation} \label{eq:examplepfequation}
\left( \left(1+\phi \right) \vartheta^3-\vartheta^2 \right)\Pi_0(\phi)=0, ~\vartheta=\phi \frac{d}{d \phi},
\end{equation}
which has three regular singularities at $\phi=0,1,\infty$. By analytic continuation, $\Pi_0$ extends to a multi-valued holomorphic function on $\mathbb{C} -\{0,1 \}$. Let us now define the variable $\varphi$ by
\begin{equation}
\varphi:=1/\phi.
\end{equation}
On the unit disc $|\varphi|<1$, the solution space of the Picard-Fuchs equation \ref{eq:examplepfequation} has a canonical basis of the form
\begin{equation}
\begin{aligned}
\varpi_0&=1, \\
\varpi_1&=\log \varphi, \\
\varpi_2&=\log^2 \varphi+2 \sum_{n=1}^{\infty} \frac{(-1)^n}{n^2}\varphi^n.
\end{aligned}
\end{equation}
In this paper, the unit circle $S^1$ will be parameterized in the way
\begin{equation} \label{eq:parameterofvarphi}
\varphi=\exp 2 \pi i\,t,~ -\frac{1}{2} < t \leq \frac{1}{2},
\end{equation}
while $\log \varphi$ defines a single-valued function on $S^1$ through
\begin{equation} \label{eq:fndefinedbylogvarphi}
\log \varphi=\log \exp 2 \pi i \,t=2 \pi i\,t,\, -\frac{1}{2} < t \leq \frac{1}{2}.
\end{equation}
With respect to the canonical basis $\{\varpi_i \}_{i=0}^2$, $\Pi_0$ has an expansion of the form
\begin{equation}
\Pi_0=\tau_0 \varpi_0+\tau_1 \varpi_1+\tau_2 \varpi_2,~\tau_i \in \mathbb{C}.
\end{equation}
We now explain how to compute the complex number $\tau_i$ using Fourier analysis. On the unit circle $S^1$, we have
\begin{equation} \label{eq:toyexamplefourierseries}
\sum_{n=1}^\infty \frac{(-1)^n}{n^2} \exp (-2 \pi i n\,t)- 2\,\tau_2\, \sum_{n=1}^{\infty} \frac{(-1)^n}{n^2}\exp (2 \pi i n\,t)=\tau_0+2 \pi i \,\tau_1 \,t+(2 \pi i)^2 \tau_2 t^2,
\end{equation}
and the LHS is just the Fourier series expansion of the RHS. Now take the integration of both sides of the above equation over $S^1$, and we have
\begin{equation}
0=\tau_0-\frac{1}{3} \pi^2 \tau_2.
\end{equation}
While the equation $\int_{-1/2}^{1/2}\phi^{-n} \text{LHS} ~dt=\int_{-1/2}^{1/2}\varphi^{n} \text{RHS} ~dt$ for various $n$, e.g. $n=1,2$ yields more linear equations about $\tau_i$. Solve these linear equations and we obtain
\begin{equation}
\tau_0=-\frac{1}{6}\pi^2,~\tau_1=0,~\tau_2=-\frac{1}{2},
\end{equation}
and the formula \ref{eq:toyexamplefourierseries} now becomes
\begin{equation}
\sum_{n=1}^\infty \frac{(-1)^n}{n^2} \cos 2 \pi n \,t=-\frac{1}{12}\pi^2+\pi^2\, t^2.
\end{equation}

\section{Double Zeta values \texorpdfstring{$\zeta(k,1)$}{z(k,1)}} \label{sec:zetak1section}

In this section, we will apply the method in Section \ref{sec:introductoryexamplepi} to study the double zeta values $\zeta(k,1)$, which by definition is given by
\begin{equation}
\zeta(k,1)=\sum_{n> m \geq 1} \frac{1}{n^k m}=\sum_{n=2}^\infty \frac{1}{n^k} \sum_{m=1}^{n-1} \frac{1}{m}=\sum_{n=1}^\infty \frac{1}{(n+1)^k} \sum_{m=1}^{n} \frac{1}{m}.
\end{equation}
The harmonic numbers $H_{n,t}$ are defined by
\begin{equation}
H_{n,t}:=\sum_{m=1}^n \frac{1}{m^t}, ~ t \in \mathbb{Z}_+,
\end{equation}
hence the double zeta value $\zeta(k,1)$ can also be written as
\begin{equation}
\zeta(k,1)=\sum_{n=1}^\infty \frac{H_{n,1}}{(n+1)^k}.
\end{equation}
Follow Section \ref{sec:introductoryexamplepi}, we construct a power series $\Pi_{k,1}$
\begin{equation}
\Pi_{k,1}:=\sum_{n=1}^\infty \frac{(-1)^n H_{n,1}}{(n+1)^k} \phi^{n+1},
\end{equation}
which converges on the unit disc $|\phi| \leq 1$, and its value at $\phi=-1$ is $-\zeta(k,1)$. 
 \begin{lemma}
 The power series $\Pi_{k,1}$ is a solution to the Picard-Fuchs operator $\mathscr{D}_{k,1}$
 \begin{equation}
\mathscr{D}_{k,1}:=(1+\phi)^2 \vartheta^{k+2}-3(1+\phi)\vartheta^{k+1}+(2+\phi)\vartheta^k, \vartheta=\phi \frac{d}{d \phi}
 \end{equation}
 \end{lemma}
\begin{proof}
The Picard-Fuchs operator $\mathscr{D}_{k,1}$ is a linear operator, and its solution space is $k+2$ dimensional. Suppose there exists a power series solution of the form
\begin{equation}
\sum_{n=2}^\infty a_n \phi^n,~\text{with}~a_2=-2^{-k}.
\end{equation}
In order for it to be a solution of the Picard-Fuchs operator $\mathscr{D}_{k,1}$, we must have
\begin{equation}
a_3=H_{2,1}\,3^{-k},~(n-1)^{k+1}a_{n-1}+n^k(2n-1)a_n+n(n+1)^ka_{n+1}=0.
\end{equation}
This recursion equation can be solved explicitly, and we obtain
\begin{equation}
a_n=\frac{(-1)^{n-1}H_{n-1,1}}{n^k},
\end{equation}
which proves this lemma.
\end{proof}
Recall that the variable $\varphi$ is defined to be $1/\phi$, with respect to which the Picard-Fuchs operator $\mathscr{D}_{k,1}$ (after a multiplication by $(-1)^{k+2}\varphi^2$) becomes
 \begin{equation}
 \mathscr{D}_{k,1}=(1+\varphi)^2 \vartheta^{k+2}+3\varphi(1+\varphi) \vartheta^{k+1}+\varphi(1+2 \varphi) \vartheta^k,~\vartheta=\varphi \frac{d }{d \varphi}.
 \end{equation}
Since the order of $\mathscr{D}_{k,1}$ is $k+2$, the dimension of its solution space is $k+2$. Now we will construct a canonical basis for it on the unit disc $|\varphi|<1$. First, from the form of $\mathscr{D}_{k,1}$, it has $k$ solutions of the form
\begin{equation}
\varpi^{k,1}_i=\log^i \varphi,~i=0,1,\cdots,k-1.
\end{equation}
We now try whether there exists a solution of the form
\begin{equation}
\log^k \varphi+\sum_{n=1}^\infty b_n\, \varphi^n,
\end{equation}
and in order for it to be a solution of $\mathscr{D}_{k,1}$, we must have
\begin{equation}
b_1=-k!,~b_2=\frac{k!}{2^k},~(n-1)^k n b_{n-1}+n^k(2n+1)b_n+(n+1)^{k+1}b_{n+1}=0.
\end{equation}
This recursion equation can be solved explicitly and we obtain
\begin{equation}
b_n=k! \frac{(-1)^n}{n^k}.
\end{equation}
Next, we try whether there exists a solution of the form
\begin{equation}
\log^{k+1} \varphi+(k+1)\left( \sum_{n=1}^\infty k! \frac{(-1)^n}{n^k} \varphi^{n} \right) \log \varphi+\sum_{n=1}^\infty c_n \varphi^n.
\end{equation}
Plug it into $\mathscr{D}_{k,1}$, we obtain a recursion equation about $c_n$
\begin{equation}
\begin{aligned}
&c_1=(k-1)(k+1)!,~c_2=-\frac{(k-3)(k+1)!}{2^{k+1}},\\
&(n+1)^{k+1}c_{n+1}+n^k(2n+1)c_n+(n-1)^knc_{n-1}+\frac{(-1)^{n+1}k (k+1)!}{n(n-1)}=0.
\end{aligned}
\end{equation}
This recursion equation can also be solved explicitly and we obtain
\begin{equation}
c_n=\frac{(k+1)!(-1)^n(-k+n H_{n,1})}{n^{k+1}}.
\end{equation}
Thus we have the following proposition.
\begin{proposition}
On the unit disc $|\varphi|<1$, the $k+2$ dimensional solution space of the Picard-Fuchs operator $\mathscr{D}_{k,1}$ has a canonical basis given by
\begin{equation}
\begin{aligned}
\varpi^{k,1}_i&=\log^i \varphi,~i=0,1,\cdots,k-1.\\
\varpi^{k,1}_{k}&=\log^k \varphi+\sum_{n=1}^\infty k! \frac{(-1)^n}{n^k}\, \varphi^n,\\
\varpi^{k,1}_{k+1}&=\log^{k+1} \varphi+(k+1)\left( \sum_{n=1}^\infty k! \frac{(-1)^n}{n^k} \varphi^{n} \right) \log \varphi+\sum_{n=1}^\infty \frac{(k+1)!(-1)^n(-k+n H_{n,1})}{n^{k+1}} \varphi^n.
\end{aligned}
\end{equation}
\end{proposition}
Since the harmonic number $H_{n,1}$ satisfy
\begin{equation}
H_{n,1}=H_{n-1,1}+\frac{1}{n},~H_{0,1}=0,
\end{equation}
the solution $\varpi^{k,1}_{k+1}$ can also be rewritten as
\begin{equation}
\begin{aligned}
\varpi^{k,1}_{k+1}=&\log^{k+1} \varphi+(k+1)\left( \sum_{n=1}^\infty k! \frac{(-1)^n}{n^k} \varphi^{n} \right) \log \varphi \\
&+(1-k)(k+1)!\sum_{n=1}^\infty \frac{(-1)^n}{n^{k+1}} \varphi^n-(k+1)! \sum_{n=1}^\infty \frac{(-1)^nH_{n,1}}{(n+1)^k} \varphi^{n+1}.
\end{aligned}
\end{equation}
But $\Pi_{k,1}$ is also a solution to the Picard-Fuchs operator $\mathscr{D}_{k,1}$, so there exists $k+2$ complex numbers $\{\tau^{k,1}_i \}_{i=0}^{k+1}$ such that
\begin{equation}  \label{eq:zeta21linearequation}
\Pi_{k,1}=\sum_{i=0}^{k+1} \tau^{k,1}_i \,\varpi^{k,1}_i,~\tau^{k,1}_i \in \mathbb{C}.
\end{equation}
The complex number $\tau^{k,1}_i$ can be computed explicitly by a similar method as in Section \ref{sec:introductoryexamplepi}. In their computations, we will need to evaluate the following infinite sum
\begin{equation}
S(m,k_1,k_2)=\sum_{n=1}^\infty \frac{1}{n^{k_1}} \frac{1}{(n+m)^{k_2}}, ~k_1,k_2,m \in \mathbb{Z}_+,
\end{equation}
which can be done recursively. More precisely, we have the following recursion relation 
\begin{equation}
\begin{aligned}
S(m,k_1,k_2)&=\sum_{n=1}^\infty \frac{1}{n^{k_1-1}} \frac{1}{(n+m)^{k_2-1}} \cdot \frac{1}{m} \left(\frac{1}{n}-\frac{1}{n+m} \right)\\
&=\frac{1}{m} \left(S(m,k_1,k_2-1)-S(m,k_1-1,k_2) \right).
\end{aligned}
\end{equation}
We can use this formula repeatedly until we arrive at three `final' cases that can be evaluated immediately
\begin{equation}
\begin{aligned}
S(m,l_1,0)&=\zeta(l_1),~l_1 \geq 2, \\
S(m,0,l_2)&=\zeta(l_2)-H_{m,l_2},~l_2 \geq 2,\\
S(m,1,1)&=\frac{1}{m}\,H_{m,1}.
\end{aligned}
\end{equation}
Let us now look at the case where $k=2$. We will show $\zeta(2,1)=\zeta(3)$, which is first proved by Euler.

\subsection{\texorpdfstring{$\bm{\zeta(2,1)}$}{z(2,1)}} \label{sec:sectionzeta21}

We first show how to compute the complex numbers $\tau^{2,1}_i$. The power series $\Pi_{2,1}$ converges absolutely on the unit circle
\begin{equation}
S^1=\{|\phi|=1 \}=\{|\varphi|=1 \},
\end{equation}
thus it defines a continuous function on it. Similarly the power series that appear in $\varpi^{2,1}_2$ and $\varpi^{2,1}_3$, i.e.
\begin{equation}
\sum_{n=1}^\infty 2 \frac{(-1)^n}{n^2}\, \varphi^n~\text{and}~\sum_{n=1}^\infty \frac{6(-1)^n(-2+n H_{n,1})}{n^3} \varphi^n,
\end{equation}
also converge absolutely on $S^1$. The unit cycle $S^1$ is parameterized in the same way as in Section \ref{sec:introductoryexamplepi}, i.e. the formula \ref{eq:parameterofvarphi}, while the function on $S^1$ defined by $\log \varphi$ is given by formula \ref{eq:fndefinedbylogvarphi}. Take the integration of the LHS and RHS of formula \ref{eq:zeta21linearequation} over $S^1$ we obtain
\begin{equation} \label{eq:zeta21phi0}
\int_{-1/2}^{1/2} \Pi_{2,1} dt=\sum_{i=0}^3 \tau^{2,1}_i \int_{-1/2}^{1/2}  \varpi^{2,1}_i dt.
\end{equation}
The integrals in this formula can be easily evaluated 
\begin{equation}
\int_{-1/2}^{1/2} \Pi_{2,1} dt=0,~ \int_{-1/2}^{1/2}  \varpi^{2,1}_0 dt=1,~\int_{-1/2}^{1/2}  \varpi^{2,1}_1 dt=0,~ \int_{-1/2}^{1/2}  \varpi^{2,1}_2 dt =-\frac{1}{3} \pi^2.
\end{equation}
Furthermore, the following integrals will be needed in this section
\begin{equation}
\begin{aligned}
\int_{-1/2}^{1/2} \varphi^n \log \varphi dt&=\frac{(-1)^n}{n},~ n\in \mathbb{Z}_+,\\
\int_{-1/2}^{1/2} \varphi^n \log^2 \varphi dt&=\frac{2(-1)^{n+1}}{n^2},~ n\in \mathbb{Z}_+,\\
\int_{-1/2}^{1/2} \varphi^n \log^3 \varphi dt&=\frac{(-1)^{n}(6-n^2\pi^2)}{n^3},~ n\in \mathbb{Z}_+.
\end{aligned}
\end{equation} 
From them we have
\begin{equation}
\int_{-1/2}^{1/2}  \varpi^{2,1}_3 dt=(2+1)\left( \sum_{n=1}^\infty 2! \frac{(-1)^n}{n^2} \right) \int_{-1/2}^{1/2}  \varphi^{n}\log \varphi dt=6 \sum_{n=1}^\infty \frac{1}{n^3}=6 \zeta(3),
\end{equation}
hence equation \ref{eq:zeta21phi0} becomes
\begin{equation}
0=\tau^{2,1}_0-\frac{1}{3}\pi^2 \tau^{2,1}_2+6 \zeta(3) \tau^{2,1}_3.
\end{equation}
Now multiply both sides of equation \ref{eq:zeta21linearequation} by $\varphi$
\begin{equation}
\int_{-1/2}^{1/2} \phi^{-1} \Pi_{2,1} dt=\sum_{i=0}^3 \tau^{2,1}_i \int_{-1/2}^{1/2} \varphi \varpi^{2,1}_i dt.
\end{equation}
The integrals in this equation are given by
\begin{equation}
\int_{-1/2}^{1/2} \phi^{-1} \Pi_{2,1} dt=0,~ \int_{-1/2}^{1/2} \varphi  \varpi^{2,1}_0 dt=0,~\int_{-1/2}^{1/2} \varphi  \varpi^{2,1}_1 dt=-1,~ \int_{-1/2}^{1/2} \varphi  \varpi^{2,1}_2 dt =2.
\end{equation}
While the integral of $\varphi \varpi^{2,1}_3$ over $S^1$ is given by
\begin{equation}
\int_{-1/2}^{1/2} \varphi  \varpi^{2,1}_3 dt=-(6-\pi^2)-6 \sum_{n=1}^\infty \frac{1}{n^2(n+1)}.
\end{equation}
The sum of the infinite series can be evaluated by
\begin{equation}
\sum_{n=1}^\infty \frac{1}{n^2(n+1)} =\sum_{n=1}^\infty \frac{1}{n}\left(\frac{1}{n}-\frac{1}{n+1} \right) =\zeta(2)-1,
\end{equation}
hence we obtain
\begin{equation}
\int_{-1/2}^{1/2} \varphi  \varpi^{2,1}_3 dt=0,
\end{equation}
which implies
\begin{equation}
-\tau^{2,1}_1+2 \tau^{2,1}_2=0.
\end{equation}
Now multiply both sides of equation \ref{eq:zeta21linearequation} by $\varphi^2$, and its integration over $S^1$ yields
\begin{equation}
-\frac{1}{4}=\frac{1}{2} \tau^{2,1}_1-\frac{1}{2} \tau^{2,1}_2-\frac{3}{2} \tau^{2,1}_3.
\end{equation} 
In order to evaluate the integral of $\varphi^2 \varpi^{2,1}_3$ over $S^1$, we have used the identity
\begin{equation}
\sum_{n=1}^\infty \frac{1}{n^2(n+2)}=\frac{1}{2} \sum_{n=1}^\infty \frac{1}{n}\left( \frac{1}{n}- \frac{1}{n+2} \right)=\frac{1}{2}\zeta(2)-\frac{3}{8}.
\end{equation}
Similarly, if we multiply both sides of equation \ref{eq:zeta21linearequation} by $\varphi^3$, and take its integration over $S^1$, we obtain
\begin{equation}
\frac{1}{6}=-\frac{1}{3} \tau^{2,1}_1+\frac{2}{9} \tau^{2,1}_2+ \tau^{2,1}_3.
\end{equation} 
In order to evaluate the integral of $\varphi^3 \varpi^{2,1}_3$ over $S^1$, we have used the identity
\begin{equation}
\sum_{n=1}^\infty \frac{1}{n^2(n+3)}=\frac{1}{3} \sum_{n=1}^\infty \frac{1}{n}\left( \frac{1}{n}- \frac{1}{n+3} \right)=\frac{1}{3}\zeta(2)-\frac{11}{54}.
\end{equation}
The solution to the four linear equations is
\begin{equation}
\tau^{2,1}_0=-\zeta(3),~\tau^{2,1}_1=0,~\tau^{2,1}_2=0,~\tau^{2,1}_3=\frac{1}{6},
\end{equation}
i.e. we have
\begin{equation} \label{eq:linearrelationkequals2}
\Pi_{2,1}(\phi)=-\zeta(3) \varpi^{2,1}_0(\varphi)+\frac{1}{6} \varpi^{2,1}_3(\varphi).
\end{equation}
\begin{remark}
The readers are referred to the paper \cite{KimYang} for more similarities between the formula \ref{eq:linearrelationkequals2} and the mirror symmetry of Calabi-Yau threefolds.
\end{remark}

Now we are ready to prove the following lemma.
\begin{lemma}
\begin{equation}
\zeta(2,1)=\sum_{n=1}^\infty \frac{H_{n,1}}{(n+1)^k}=\zeta(3),~\sum_{n=1}^\infty \frac{(-1)^nH_{n,1}}{(n+1)^k}=-\frac{1}{8} \,\zeta(3).
\end{equation}
\end{lemma}
\begin{proof}
In the formula \ref{eq:linearrelationkequals2}, let $\phi=-1$ ($\varphi=-1$) and we obtain
\begin{equation} \label{eq:zeta21equationminus1}
\Pi_{2,1}(-1)=-\zeta(3)+\frac{1}{6}\varpi^{2,1}_3(-1).
\end{equation}
The value of $\Pi_{2,1}$ at $\phi=-1$ is 
\begin{equation}
\Pi_{2,1}(-1)=\sum_{n=1}^\infty \frac{(-1)^n H_{n,1}}{(n+1)^2} (-1)^{n+1}=-\zeta(2,1),
\end{equation}
while the value of $\varpi^{2,1}_3$ at $\varphi=-1$ is given by
\begin{equation}
\varpi^{2,1}_3(-1)=6 \sum_{n=1}^\infty \frac{-2+nH_{n,1}}{n^3}=6 \sum_{n=1}^\infty \frac{-1+nH_{n-1,1}}{n^3}=6\left( -\zeta(3)+\zeta(2,1) \right).
\end{equation}
Plug the values of $\Pi_{2,1}(-1)$ and $\varpi^{2,1}_3(-1)$ into the equation \ref{eq:zeta21equationminus1} we get
\begin{equation}
\zeta(2,1)=\zeta(3).
\end{equation}
Similarly, let $\phi=1$ ($\varphi=1$) in the formula \ref{eq:linearrelationkequals2}, and we obtain
\begin{equation} \label{eq:zeta21equation1}
\Pi_{2,1}(1)=-\zeta(3)+\frac{1}{6}\varpi^{2,1}_3(1).
\end{equation}
The value of $\Pi_{2,1}$ at $\phi=1$ is 
\begin{equation}
\Pi_{2,1}(1)=\sum_{n=1}^\infty \frac{(-1)^n H_{n,1}}{(n+1)^2},
\end{equation}
while the value of $\varpi^{2,1}_3$ at $\varphi=1$ is given by
\begin{equation}
\varpi^{2,1}_3(1)=6 \sum_{n=1}^\infty \frac{(-1)^n(-2+nH_{n,1})}{n^3}=-6 \sum_{n=1}^\infty \frac{(-1)^n}{n^3}-6 \sum_{n=1}^\infty \frac{(-1)^nH_{n,1}}{(n+1)^2}
\end{equation}
The sum $\sum_{n=1}^\infty (-1)^n/n^3$ can be evaluated in an elementary way
\begin{equation}
\sum_{n=1}^\infty \frac{(-1)^n}{n^3}+\sum_{n=1}^\infty \frac{1}{n^3}=2 \sum_{n=1}^{\infty} \frac{1}{(2n)^3}=\frac{1}{4} \zeta(3),
\end{equation}
hence we deduce 
\begin{equation}
\sum_{n=1}^\infty \frac{(-1)^n}{n^3}=-\frac{3}{4} \zeta(3).
\end{equation}
Therefore we have 
\begin{equation}
\sum_{n=1}^\infty \frac{(-1)^n H_{n,1}}{(n+1)^2}=-\frac{1}{8} \zeta(3).
\end{equation}
\end{proof}

\subsection{\texorpdfstring{$\bm{\zeta(3,1)}$}{z(3,1)}}

We now look at the case where $k=3$. The value of $\tau^{3,1}_i$ can be computed by the same method as in Section \ref{sec:sectionzeta21}. Namely we look at the equations given by
\begin{equation}
\int_{-1/2}^{1/2} \phi^{-n} \Pi_{3,1} dt=\sum_{i=0}^4 \tau^{3,1}_i \int_{-1/2}^{1/2} \varphi^n \varpi^{3,1}_i dt,
\end{equation}
for various $n$, e.g. $n=0,1,2,3,4$, and then we solve these linear equations. In order to evaluate the integral $\int_{-1/2}^{1/2} \varphi^n \varpi^{3,1}_i dt$, we will further need
\begin{equation}
\int_{-1/2}^{1/2} \varphi^n \log^4 \varphi dt=4 \frac{(-1)^{n}(-6+n^2\pi^2)}{n^4},~ n\in \mathbb{Z}_+.
\end{equation}
The result is
\begin{equation}
\tau^{3,1}_0=\frac{7}{4} \zeta(4),~\tau^{3,1}_1=\zeta(3),~\tau^{3,1}_2=\tau^{3,1}_3=0,~\tau^{3,1}_4=-\frac{1}{4!}.
\end{equation}
However, if we let $\phi=1$ ($\varphi=1$), the equation
\begin{equation}
\Pi_{3,1}(1)=\sum_{i=0}^4 \tau^{3,1}_i \varpi^{3,1}_i(1)
\end{equation}
becomes the trivial one $0=0$. Similarly, the equation
\begin{equation}
\Pi_{3,1}(-1)=\sum_{i=0}^4 \tau^{3,1}_i \varpi^{3,1}_i(-1)
\end{equation}
also yields $0=0$. Thus the value of $\zeta(3,1)$ can not be determined by this method.

\subsection{\texorpdfstring{$\bm{\zeta(4,1)}$}{z(4,1)}}

When $k=4$, in order to evaluate $\int_{-1/2}^{1/2} \varphi^n \varpi^{4,1}_i dt$, we will further need the integral
\begin{equation}
\int_{-1/2}^{1/2} \varphi^n \log^5 \varphi dt= \frac{(-1)^{n}(120-20 n^2\pi^2+n^4 \pi^4)}{n^5},~ n\in \mathbb{Z}_+.
\end{equation}
Use the same method as in Section \ref{sec:sectionzeta21}, we obtain
\begin{equation}
\tau^{4,1}_0=-\zeta(5)-\zeta(2)\zeta(3),~\tau^{4,1}_1=-\frac{7}{4} \zeta(4),~\tau^{4,1}_2=-\frac{1}{2}\zeta(3),~\tau^{4,1}_3=\tau^{4,1}_4=0,~\tau^{4,1}_5=\frac{1}{5!}.
\end{equation}
Now let $\phi=1$ ($\varphi=1$) and $\phi=-1$ ($\varphi=-1$) respectively, the equations
\begin{equation}
\Pi_{4,1}(1)=\sum_{i=0}^5 \tau^{4,1}_i \varpi^{4,1}_i(1)~\text{and}~\Pi_{4,1}(-1)=\sum_{i=0}^5 \tau^{4,1}_i \varpi^{4,1}_i(-1)
\end{equation}
give us the following lemma.
\begin{lemma}
\begin{equation}
\zeta(4,1)=\sum_{n=1}^\infty \frac{H_{n,1}}{(n+1)^4}=2\zeta(5)-\zeta(2)\zeta(3),~\sum_{n=1}^\infty \frac{(-1)^nH_{n,1}}{(n+1)^4}=\frac{29}{32} \zeta(5)-\frac{1}{2}\zeta(2) \zeta(3).
\end{equation}
\end{lemma}

\subsection{\texorpdfstring{$\bm{\zeta(5,1)}$}{z(5,1)}}

When $k=5$, in order to evaluate $\int_{-1/2}^{1/2} \varphi^n \varpi^{5,1}_i dt$, we will further need the integral
\begin{equation}
\int_{-1/2}^{1/2} \varphi^n \log^6 \varphi dt= -\frac{6 (-1)^{n}(120-20 n^2\pi^2+n^4 \pi^4)}{n^6},~ n\in \mathbb{Z}_+.
\end{equation}
Use the same method as in Section \ref{sec:sectionzeta21}, we obtain
\begin{equation}
\begin{aligned}
&\tau^{5,1}_0=\frac{31}{8} \zeta(6),~\tau^{5,1}_1=\zeta(5)+\zeta(2)\zeta(3),~\tau^{5,1}_2=\frac{7}{8} \zeta(4),\\
&\tau^{5,1}_3=\frac{1}{6}\zeta(3),~\tau^{5,1}_4=\tau^{5,1}_5=0,~\tau^{5,1}_6=-\frac{1}{6!}.
\end{aligned}
\end{equation}
Now let $\phi=1$ ($\varphi=1$) and $\phi=-1$ ($\varphi=-1$) respectively, the equations
\begin{equation}
\Pi_{5,1}(1)=\sum_{i=0}^6 \tau^{5,1}_i \varpi^{5,1}_i(1)~\text{and}~\Pi_{5,1}(-1)=\sum_{i=0}^6 \tau^{5,1}_i \varpi^{5,1}_i(-1)
\end{equation}
give us the trivial identity $0=0$. 

\subsection{\texorpdfstring{$\bm{\zeta(6,1)}$}{z(6,1)}}

When $k=6$, in order to evaluate $\int_{-1/2}^{1/2} \varphi^n \varpi^{6,1}_i dt$, we will further need the integral
\begin{equation}
\int_{-1/2}^{1/2} \varphi^n \log^7 \varphi dt= \frac{ (-1)^{n}(5040-840 n^2\pi^2+42 n^4 \pi^4-n^6 \pi^6)}{n^7},~ n\in \mathbb{Z}_+.
\end{equation}
Use the same method as in Section \ref{sec:sectionzeta21}, we obtain
\begin{equation}
\begin{aligned}
&\tau^{6,1}_0=-\zeta(7)-\zeta(2)\zeta(5)-\frac{7}{4}\zeta(3) \zeta(4),~\tau^{6,1}_1=-\frac{31}{8} \zeta(6),~\tau^{6,1}_2=-\frac{1}{2}(\zeta(5)+\zeta(2)\zeta(3)),\\
&\tau^{6,1}_3=-\frac{7}{24} \zeta(4),~\tau^{6,1}_4=-\frac{1}{24}\zeta(3),~\tau^{6,1}_5=\tau^{6,1}_6=0,~\tau^{6,1}_7=\frac{1}{7!}.
\end{aligned}
\end{equation}
Now let $\phi=1$ ($\varphi=1$) and $\phi=-1$ ($\varphi=-1$) respectively, the equations
\begin{equation}
\Pi_{6,1}(1)=\sum_{i=0}^6 \tau^{6,1}_i \varpi^{6,1}_i(1)~\text{and}~\Pi_{6,1}(-1)=\sum_{i=0}^6 \tau^{6,1}_i \varpi^{6,1}_i(-1)
\end{equation}
give us the following lemma.
\begin{lemma}
\begin{equation}
\begin{aligned}
&\zeta(6,1)=\sum_{n=1}^\infty \frac{H_{n,1}}{(n+1)^6}=3\zeta(7)-\zeta(2)\zeta(5)-\zeta(3)\zeta(4), \\
&\sum_{n=1}^\infty \frac{(-1)^nH_{n,1}}{(n+1)^6}=\frac{251}{128} \zeta(7)-\frac{1}{2}\zeta(2) \zeta(5)-\frac{7}{8} \zeta(3)\zeta(4).
\end{aligned}
\end{equation}
\end{lemma}

\subsection{\texorpdfstring{$\bm{\zeta(7,1)}$}{z(7,1)}}

When $k=7$, in order to evaluate $\int_{-1/2}^{1/2} \varphi^n \varpi^{7,1}_i dt$, we will further need the integral
\begin{equation}
\int_{-1/2}^{1/2} \varphi^n \log^{8} \varphi dt= \frac{8 (-1)^{n}(-5040+840 n^2\pi^2-42n^4 \pi^4+n^6\pi^6)}{n^{8}},~ n\in \mathbb{Z}_+.
\end{equation}
Use the same method as in Section \ref{sec:sectionzeta21}, we obtain
\begin{equation}
\begin{aligned}
&\tau^{7,1}_0=\frac{381}{64} \zeta(8),~\tau^{7,1}_1=\zeta(7)+\zeta(2)\zeta(5)+\frac{7}{4}\zeta(3) \zeta(4),~\tau^{7,1}_2=\frac{31}{16} \zeta(6),\\
&\tau^{9,1}_5=\frac{1}{6} ( \zeta(5)+\zeta(2)\zeta(3)), \tau^{7,1}_4=\frac{7}{96} \zeta(4),\tau^{7,1}_5=\frac{1}{120}\zeta(3),\tau^{7,1}_6=\tau^{7,1}_7=0,\tau^{7,1}_{8}=-\frac{1}{8!}.
\end{aligned}
\end{equation}
Now let $\phi=1$ ($\varphi=1$) and $\phi=-1$ ($\varphi=-1$) respectively, the equations
\begin{equation}
\Pi_{7,1}(1)=\sum_{i=0}^{8} \tau^{7,1}_i \varpi^{7,1}_i(1)~\text{and}~\Pi_{7,1}(-1)=\sum_{i=0}^{8} \tau^{7,1}_i \varpi^{7,1}_i(-1)
\end{equation}
give us the trivial identity $0=0$. 

\subsection{\texorpdfstring{$\bm{\zeta(8,1)}$}{z(8,1)}}

When $k=8$, in order to evaluate $\int_{-1/2}^{1/2} \varphi^n \varpi^{8,1}_i dt$, we will further need the integral
\begin{equation}
\int_{-1/2}^{1/2} \varphi^n \log^{9} \varphi dt= \frac{ (-1)^{n}(362\,880-60\,480 n^2\pi^2+3\,024n^4 \pi^4-72n^6\pi^6 +n^8 \pi^8)}{n^{9}},~ n\in \mathbb{Z}_+.
\end{equation}
Use the same method as in Section \ref{sec:sectionzeta21}, we obtain
\begin{equation}
\begin{aligned}
&\tau^{8,1}_0=-\zeta(9)-\zeta(2)\zeta(7)-\frac{31}{16}\zeta(3)\zeta(6)-\frac{7}{4}\zeta(4)\zeta(5),~\tau^{8,1}_1=-\frac{381}{64} \zeta(8), \\
&\tau^{8,1}_2=-\frac{1}{2}\left(\zeta(7)+\zeta(2)\zeta(5)+\frac{7}{4} \zeta(3)\zeta(4)\right),\tau^{8,1}_3=-\frac{31}{48} \zeta(6),\tau^{8,1}_4=-\frac{1}{24} ( \zeta(5)+\zeta(2)\zeta(3)), \\
&\tau^{8,1}_5=-\frac{7}{480} \zeta(4),~\tau^{8,1}_6=-\frac{1}{720}\zeta(3),~\tau^{8,1}_7=\tau^{8,1}_8=0,~\tau^{8,1}_{9}=\frac{1}{9!}.
\end{aligned}
\end{equation}
Now let $\phi=1$ ($\varphi=1$) and $\phi=-1$ ($\varphi=-1$) respectively, the equations
\begin{equation}
\Pi_{8,1}(1)=\sum_{i=0}^{9} \tau^{8,1}_i \varpi^{8,1}_i(1)~\text{and}~\Pi_{8,1}(-1)=\sum_{i=0}^{9} \tau^{8,1}_i \varpi^{8,1}_i(-1)
\end{equation}
give us the following lemma.
\begin{lemma}
\begin{equation}
\begin{aligned}
&\zeta(8,1)=\sum_{n=1}^\infty \frac{H_{n,1}}{(n+1)^8}=4\zeta(9)-\zeta(2)\zeta(7)-\zeta(3)\zeta(6)-\zeta(4)\zeta(5), \\
&\sum_{n=1}^\infty \frac{(-1)^nH_{n,1}}{(n+1)^8}=\frac{1529}{512} \zeta(9)-\frac{1}{2}\zeta(2) \zeta(7)-\frac{31}{32}\zeta(3)\zeta(6)-\frac{7}{8} \zeta(4)\zeta(5).
\end{aligned}
\end{equation}
\end{lemma}

\subsection{\texorpdfstring{$\bm{\zeta(9,1)}$}{z(9,1)}}

When $k=9$, in order to evaluate $\int_{-1/2}^{1/2} \varphi^n \varpi^{9,1}_i dt$, we will further need the integral
\begin{equation}
\int_{-1/2}^{1/2} \varphi^n \log^{10} \varphi dt= -\frac{10 (-1)^{n}(362880-60480 n^2\pi^2+3024n^4 \pi^4-72n^6\pi^6 +n^8 \pi^8)}{n^{10}},~ n\in \mathbb{Z}_+.
\end{equation}
Use the same method as in Section \ref{sec:sectionzeta21}, we obtain
\begin{equation}
\begin{aligned}
&\tau^{9,1}_0=\frac{511}{64} \zeta(10),~\tau^{9,1}_1=\zeta(9)+\zeta(2)\zeta(7)+\frac{31}{16}\zeta(3)\zeta(6)+\frac{7}{4}\zeta(4)\zeta(5),~\tau^{9,1}_2=\frac{381}{128} \zeta(8), \\
&\tau^{9,1}_3=\frac{1}{6}\left(\zeta(7)+\zeta(2)\zeta(5)+\frac{7}{4} \zeta(3)\zeta(4)\right),~\tau^{9,1}_4=\frac{31}{192} \zeta(6),~\tau^{9,1}_5=\frac{1}{120} ( \zeta(5)+\zeta(2)\zeta(3)), \\
&\tau^{9,1}_6=\frac{7}{2880} \zeta(4),~\tau^{9,1}_7=\frac{1}{5040}\zeta(3),~\tau^{9,1}_8=\tau^{9,1}_9=0,~\tau^{9,1}_{10}=-\frac{1}{10!}.
\end{aligned}
\end{equation}
Now let $\phi=1$ ($\varphi=1$) and $\phi=-1$ ($\varphi=-1$) respectively, the equations
\begin{equation}
\Pi_{9,1}(1)=\sum_{i=0}^{10} \tau^{9,1}_i \varpi^{9,1}_i(1)~\text{and}~\Pi_{9,1}(-1)=\sum_{i=0}^{10} \tau^{9,1}_i \varpi^{9,1}_i(-1)
\end{equation}
give us the trivial identity $0=0$.

\subsection{\textbf{Generalization}}

As the readers might have noticed, for large $k$, it is practically very difficult to solve the linear equations about $\tau^{k,1}_i$ given by 
\begin{equation} \label{eq:zeta21general}
\int_{-1/2}^{1/2} \phi^{-n} \Pi_{k,1} dt=\sum_{i=0}^{k+1} \tau^{k,1}_i \,\int_{-1/2}^{1/2} \varphi^{n} \varpi^{k,1}_i dt,~\tau^{k,1}_i \in \mathbb{C}.
\end{equation}
However, there is one critical observation that will make life much easier. From our computations of $\tau^{k,1}_i$ when $k=2,3,4,5,6,7,8,9$, we have observed that
\begin{equation}
\tau^{k,1}_i=-\frac{1}{i}\tau^{k-1,1}_{i-1},~i \geq 1.
\end{equation}
On the other hand, let $n=0$ in the formula \ref{eq:zeta21general}, and we have
\begin{equation} \label{eq:tauk1zeroequation}
\tau^{k,1}_0=- \sum_{i=1}^{k+1} \tau^{k,1}_i \int_{-1/2}^{1/2}  \varpi^{k,1}_i dt,
\end{equation}
where we have used
\begin{equation}
\int_{-1/2}^{1/2}  \Pi_{k,1} dt=0~\text{and}~\int_{-1/2}^{1/2}   \varpi^{k,1}_0 dt=1.
\end{equation}
The other integrals in the formula \ref{eq:tauk1zeroequation} can also be evaluated easily
\begin{equation}
\begin{aligned}
\int_{-1/2}^{1/2}  \varpi^{k,1}_j dt&=\frac{\left(1+(-1)^j \right) (\pi i)^j}{2(1+j)},~j=1,\cdots,k;\\
\int_{-1/2}^{1/2}  \varpi^{k,1}_{k+1} dt&=\frac{\left(1+(-1)^{k+1} \right) (\pi i)^{k+1}}{2(k+2)}+(k+1)!\zeta(k+1).\\
\end{aligned}
\end{equation}

\begin{conjecture} \label{conjecturetauk1}
The complex number $\tau^{k,1}_i,i \geq 1$ is always equal to $-\tau^{k-1,1}_{i-1}/i$. Together with formula \ref{eq:tauk1zeroequation}, it gives us a very efficient algorithm to compute $\tau^{k,1}_i$.
\end{conjecture}

\noindent Moreover, we have the following corollary.
\begin{corollary}
The complex number $\tau^{k,1}_{k+1}$ is equal to $(-1)^{k}/(k+1)!$. When $k$ is an even integer, the equations 
\begin{equation}
\Pi_{k,1}(1)=\sum_{i=0}^{k+1} \tau^{k,1}_i \varpi^{k,1}_i(1)~\text{and}~\Pi_{k,1}(-1)=\sum_{i=0}^{k+1} \tau^{k,1}_i \varpi^{k,1}_i(-1)
\end{equation}
will give us the values of
\begin{equation}
\zeta(k,1)=\sum_{n=1}^\infty \frac{H_{n,1}}{(n+1)^k}~\text{and}~\sum_{n=1}^\infty \frac{(-1)^nH_{n,1}}{(n+1)^k}
\end{equation}
in terms of the zeta values $\zeta(2),\cdots,\zeta(k+1)$. While when $k$ is odd, these two equations give us the trivial identity $0=0$.

\end{corollary}

\section{Double Zeta values \texorpdfstring{$ \zeta(k,2)$}{z(k,2)}} \label{sec:doublezetak2}

In this section, we will apply the method in Section \ref{sec:zetak1section} to study the double zeta values $\zeta(k,2),k\geq 2$ defined by
\begin{equation}
\zeta(k,2)=\sum_{n> m \geq 1} \frac{1}{n^k m^2}=\sum_{n=2}^\infty \frac{1}{n^k} \sum_{m=1}^{n-1} \frac{1}{m^2}=\sum_{n=1}^\infty \frac{H_{n,2}}{(n+1)^k}.
\end{equation}
Follow Section \ref{sec:zetak1section}, we construct a power series $\Pi_{k,2}$ of the form
\begin{equation}
\Pi_{k,2}:=\sum_{n=1}^\infty \frac{(-1)^n H_{n,2}}{(n+1)^k} \phi^{n+1},
\end{equation}
which converges on the unit disc $|\phi| \leq 1$, while its value at $\phi=-1$ is just $-\zeta(k,2)$. 

\begin{lemma}
The power series $\Pi_{k,2}$  is a solution to the Picard-Fuchs operator $\mathscr{D}_{(k,2)}$
\begin{equation}
\mathscr{D}_{(k,2)}:=(1+\phi)^2 \vartheta^{k+3}-4(1+\phi)\vartheta^{k+2}+(5+3\phi)\vartheta^{k+1}-(2+\phi)\vartheta^k, ~\vartheta=\phi \frac{d}{d \phi}
\end{equation}
\end{lemma}
\begin{proof}
The Picard-Fuchs operator $\mathscr{D}_{k,2}$ is a linear operator, and its solution space is $k+3$ dimensional. Suppose there exists a power series solution of the form
\begin{equation}
\sum_{n=2}^\infty a_n \phi^n,~\text{with}~a_2=-2^{-k}.
\end{equation}
Now we plug it into $\mathscr{D}_{k,2}$, and in order for it to be a solution, we must have
\begin{equation}
a_3=H_{2,2}\,3^{-k},~(n-1)^{k+2}a_{n-1}+n^k(2n^2-2n+1)a_n+n^2(n+1)^k a_{n+1}=0.
\end{equation}
This recursion relation can be solved explicitly, and we have
\begin{equation}
a_n=\frac{(-1)^{n-1}H_{n-1,2}}{n^k},
\end{equation}
which proves this lemma.
\end{proof}

With respect to the variable $\varphi=1/\phi$, the Picard-Fuchs operator $\mathscr{D}_{k,2}$ (after a multiplication by $(-1)^{k+3}\varphi^2$) becomes
 \begin{equation}
 \mathscr{D}_{k,2}=(1+\varphi)^2 \vartheta^{k+3}+4\varphi(1+\varphi) \vartheta^{k+2}+\varphi(3+5\varphi)\vartheta^{k+1}+\varphi(1+2 \varphi) \vartheta^k,~\vartheta=\varphi \frac{d }{d \varphi}.
 \end{equation}
Since the degree of $\mathscr{D}_{k,2}$ is $k+3$, the dimension of its solution space is $k+3$, and now we will construct a canonical basis for the solution space of $\mathscr{D}_{k,2}$ on the unit disc $|\varphi|<1$. First, from the form of $\mathscr{D}_{k,2}$, it has $k$ solutions of the form
\begin{equation}
\varpi^{k,2}_i=\log^i \varphi,~i=0,1,\cdots,k-1.
\end{equation}
We need to construct another three linearly independent solutions. First, let us try whether there exists a solution of the form
\begin{equation}
\log^k \varphi+\sum_{n=1}^\infty b_n\, \varphi^n.
\end{equation}
In order for it to be a solution of $\mathscr{D}_{k,2}$, we must have
\begin{equation}
b_1=-k!,~b_2=\frac{k!}{2^{k}},~(n-1)^k n^2 b_{n-1}+n^k(2n^2+2n+1)b_n+(n+1)^{k+2}b_{n+1}=0.
\end{equation}
This recursion equation can be solved explicitly, and we obtain
\begin{equation}
b_n=k! \frac{(-1)^n}{n^k}.
\end{equation}
Next, we try whether there exists a solution of the form
\begin{equation}
\log^{k+1} \varphi+(k+1)\left( \sum_{n=1}^\infty k! \frac{(-1)^n}{n^k} \varphi^{n} \right) \log \varphi+\sum_{n=1}^\infty c_n \varphi^n.
\end{equation}
Plug it into $\mathscr{D}_{k,2}$, we obtain a recursion equation about $c_n$
\begin{equation}
\begin{aligned}
&c_1=k(k+1)!,~c_2=-\frac{k(k+1)!}{2^{k+1}},\\
&(n+1)^{k+2}c_{n+1}+n^k(2n^2+2n+1)c_n+(n-1)^kn^2c_{n-1}+\frac{(-1)^{n+1}k (k+1)!}{n(n-1)}=0.
\end{aligned}
\end{equation}
This recursion equation can also be solved explicitly and we obtain
\begin{equation}
c_n=\frac{-k(k+1)!(-1)^n}{n^{k+1}}.
\end{equation}
Next, let us try whether there exists a solution of the form
\begin{equation}
\begin{aligned}
&\log^{k+2} \varphi+\binom{k+2}{2}\left( \sum_{n=1}^\infty k! \frac{(-1)^n}{n^k} \varphi^{n} \right) \log^2 \varphi\\
&+(k+2) \left( \sum_{n=1}^\infty (-k(k+1)!) \frac{(-1)^n}{n^{k+1}}  \varphi^n \right) \log \varphi+\sum_{n=1}^\infty d_n \varphi^n.
\end{aligned}
\end{equation}
Plug it into $\mathscr{D}_{k,2}$, we obtain
\begin{equation}
\begin{aligned}
&d_1=-\frac{1}{2}(k^2+k+2)(k+2)!,~d_2=\frac{(k^2+k+10)(k+2)!}{2^{k+3}},\\
&(n+1)^{k+2}d_{n+1}+n^k(2n^2+2n+1)d_n+(n-1)^kn^2d_{n-1} \\
&+\frac{(-1)^{n}k(k+1) (k+2)!(2n^2-1)}{2n^2(n-1)^2}=0.
\end{aligned}
\end{equation}
This recursion equation can be solved explicitly and we get
\begin{equation}
d_n=(k+2)!(-1)^n \frac{\binom{k+1}{2}+n^2 H_{n,2}}{n^{k+2}}.
\end{equation}
In conclusion, we have the following proposition.
\begin{proposition}
On the unit disc $|\varphi|<1$, the $k+3$ dimensional solution space of the Picard-Fuchs operator $\mathscr{D}_{k,2}$ has a canonical basis given by
\begin{equation}
\begin{aligned}
\varpi^{k,2}_i&=\log^i \varphi,~i=0,1,\cdots,k-1.\\
\varpi^{k,2}_{k}&=\log^k \varphi+\sum_{n=1}^\infty k! \frac{(-1)^n}{n^k}\, \varphi^n,\\
\varpi^{k,2}_{k+1}&=\log^{k+1} \varphi+(k+1)\left( \sum_{n=1}^\infty k! \frac{(-1)^n}{n^k} \varphi^{n} \right) \log \varphi+\sum_{n=1}^\infty (-k(k+1)!)\frac{(-1)^n}{n^{k+1}} \varphi^n,\\
\varpi^{k,2}_{k+2}&=\log^{k+2} \varphi+\binom{k+2}{2}\left( \sum_{n=1}^\infty k! \frac{(-1)^n}{n^k} \varphi^{n} \right) \log^2 \varphi\\
&+(k+2)\left( \sum_{n=1}^\infty (-k(k+1)!) \frac{(-1)^n}{n^{k+1}}  \varphi^n \right) \log \varphi+\sum_{n=1}^\infty (k+2)!(-1)^n \frac{\binom{k+1}{2}+n^2 H_{n,2}}{n^{k+2}} \varphi^n.
\end{aligned}
\end{equation}
\end{proposition}

Since the harmonic number $H_{n,2}$ satisfies
\begin{equation}
H_{n,2}=H_{n-1,2}+\frac{1}{n^2},~H_{0,2}=0,
\end{equation}
the solution $\varpi^{k,2}_{k+2}$ can also be written as
\begin{equation}
\begin{aligned}
&\log^{k+2} \varphi+\binom{k+2}{2}\left( \sum_{n=1}^\infty k! \frac{(-1)^n}{n^k} \varphi^{n} \right) \log^2 \varphi+(k+2)\left( \sum_{n=1}^\infty (-k(k+1)!) \frac{(-1)^n}{n^{k+1}}  \varphi^n \right) \log \varphi\\
&+\sum_{n=1}^\infty (k+2)!(-1)^n \frac{\binom{k+1}{2}+1}{n^{k+2}} \varphi^n-\sum_{n=1}^\infty (k+2)! \frac{ (-1)^{n} H_{n,2}}{(n+1)^{k}} \varphi^{n+1}.
\end{aligned}
\end{equation}
Since $\Pi_{k,2}$ is also a solution of the Picard-Fuchs operator $\mathscr{D}_{k,2}$, there exists $k+3$ complex numbers $\{\tau^{k,2}_i \}_{i=0}^{k+2}$ such that
\begin{equation}
\Pi_{k,2}=\sum_{i=0}^{k+2} \tau^{k,2}_i \,\varpi^{k,2}_i.
\end{equation}
The complex numbers $ \tau^{k,2}_i $ can be computed by the same method as in Section \ref{sec:sectionzeta21}. More precisely, we have linear equations about $ \tau^{k,2}_i $
\begin{equation}
\int_{-1/2}^{1/2} \phi^{-n} \Pi_{k,2} dt=\sum_{i=0}^{k+2} \tau^{k,2}_i \int_{-1/2}^{1/2} \varphi^n \varpi^{k,2}_i dt,
\end{equation}
for various $n$, e.g. $n=0,1,\cdots,k+2$. Solve these linear equations we get the value of $\tau^{k,2}_i$, therefore we have the following lemma.
\begin{lemma}
The complex numbers $\tau^{k,2}_i$ lie in the field $\mathbb{Q}(\pi,\zeta(3),\cdots, \zeta(k+2))$.
\end{lemma}

\noindent We now look at the case where $k=2$.

\subsection{\texorpdfstring{$\bm{\zeta(2,2)}$}{z(2,2)}}

When $k=2$, use the same method as in Section \ref{sec:sectionzeta21}, we obtain
\begin{equation}
\tau^{2,2}_0=-\frac{19}{4} \zeta(4),~\tau^{2,2}_1=-2\zeta(3),~\tau^{2,2}_2=-\frac{1}{2}\zeta(2),~\tau^{2,2}_3=0,~\tau^{2,2}_4=-\frac{1}{4!}.
\end{equation}
The equations
\begin{equation}
\Pi_{2,2}(1)=\sum_{i=0}^4 \tau^{2,2}_i \varpi^{2,2}_i(1)~\text{and}~\Pi_{2,2}(-1)=\sum_{i=0}^4 \tau^{2,2}_i \varpi^{2,2}_i(-1)
\end{equation}
give us the trivial identity $0=0$.

\subsection{\texorpdfstring{$\bm{\zeta(3,2)}$}{z(3,2)}}

When $k=3$, use the same method as in Section \ref{sec:sectionzeta21}, we obtain
\begin{equation}
\begin{aligned}
&\tau^{3,2}_0=4\zeta(5)+2\zeta(2)\zeta(3),\tau^{3,2}_1=\frac{19}{4} \zeta(4),~\tau^{3,2}_2=\zeta(3),\\
&\tau^{3,2}_3=\frac{1}{6}\zeta(2),\tau^{3,2}_4=0,~\tau^{3,2}_5=\frac{1}{5!}.
\end{aligned}
\end{equation}
The equations
\begin{equation}
\Pi_{3,2}(1)=\sum_{i=0}^5 \tau^{3,2}_i \varpi^{3,2}_i(1)~\text{and}~\Pi_{3,2}(-1)=\sum_{i=0}^5 \tau^{3,2}_i \varpi^{3,2}_i(-1)
\end{equation}
give us the following lemma.
\begin{lemma}
\begin{equation}
\zeta(3,2)=\sum_{n=1}^\infty \frac{H_{n,2}}{(n+1)^3}=-\frac{11}{2} \zeta(5)+3\zeta(2)\zeta(3),~\sum_{n=1}^\infty \frac{(-1)^nH_{n,2}}{(n+1)^3}=-\frac{41}{32} \zeta(5)+\frac{5}{8}\zeta(2) \zeta(3).
\end{equation}
\end{lemma}

\subsection{\texorpdfstring{$\bm{\zeta(4,2)}$}{z(4,2)}}

When $k=4$, use the same method as in Section \ref{sec:sectionzeta21}, we obtain
\begin{equation}
\begin{aligned}
&\tau^{4,2}_0=-\frac{195}{16}\zeta(6),~\tau^{4,2}_1=-4\zeta(5)-2\zeta(2)\zeta(3),~\tau^{4,2}_2=-\frac{19}{8} \zeta(4),\\
&\tau^{4,2}_3=-\frac{1}{3}\zeta(3),~\tau^{4,2}_4=-\frac{1}{24}\zeta(2),~\tau^{4,2}_5=0,~\tau^{4,2}_6=-\frac{1}{6!}.
\end{aligned}
\end{equation}
The equations
\begin{equation}
\Pi_{4,2}(1)=\sum_{i=0}^6 \tau^{4,2}_i \varpi^{4,2}_i(1)~\text{and}~\Pi_{4,2}(-1)=\sum_{i=0}^6 \tau^{4,2}_i \varpi^{4,2}_i(-1)
\end{equation}
give us the trivial identity $0=0$.

\subsection{\texorpdfstring{$\bm{\zeta(5,2)}$}{z(5,2)}}

When $k=5$, use the same method as in Section \ref{sec:sectionzeta21}, we obtain
\begin{equation}
\begin{aligned}
&\tau^{5,2}_0=6\zeta(7)+4\zeta(2)\zeta(5)+\frac{7}{2}\zeta(3)\zeta(4),~\tau^{5,2}_1=\frac{195}{16}\zeta(6),~\tau^{5,2}_2=2 \zeta(5)+\zeta(2)\zeta(3),\\
&\tau^{5,2}_3=\frac{19}{24} \zeta(4),\tau^{5,2}_4=\frac{1}{12}\zeta(3),~\tau^{5,2}_5=\frac{1}{120}\zeta(2),~\tau^{5,2}_6=0,~\tau^{5,2}_7=\frac{1}{7!}.
\end{aligned}
\end{equation}
The equations
\begin{equation}
\Pi_{5,2}(1)=\sum_{i=0}^7 \tau^{5,2}_i \varpi^{5,2}_i(1)~\text{and}~\Pi_{5,2}(-1)=\sum_{i=0}^7 \tau^{5,2}_i \varpi^{5,2}_i(-1)
\end{equation}
give us the following lemma.
\begin{lemma}
\begin{equation}
\begin{aligned}
&\zeta(5,2)=\sum_{n=1}^\infty \frac{H_{n,2}}{(n+1)^5}=-11 \zeta(7)+5\zeta(2)\zeta(5)+2\zeta(3)\zeta(4), \\
&\sum_{n=1}^\infty \frac{(-1)^nH_{n,2}}{(n+1)^5}=-\frac{39}{8} \zeta(7)+\frac{49}{32}\zeta(2) \zeta(5)+\frac{7}{4} \zeta(3)\zeta(4).
\end{aligned}
\end{equation}
\end{lemma}

\subsection{\texorpdfstring{$\bm{\zeta(6,2)}$}{z(6,2)}}

When $k=6$, use the same method as in Section \ref{sec:sectionzeta21}, we obtain
\begin{equation}
\begin{aligned}
&\tau^{6,2}_0=-\frac{4501}{192}\zeta(8),\tau^{6,2}_1=-6\zeta(7)-4\zeta(2)\zeta(5)-\frac{7}{2}\zeta(3)\zeta(4),\tau^{6,2}_2=-\frac{195}{32}\zeta(6),\\
&\tau^{6,2}_3=-\frac{2}{3} \zeta(5)-\frac{1}{3}\zeta(2)\zeta(3),\tau^{6,2}_4=-\frac{19}{96} \zeta(4),\tau^{6,2}_5=-\frac{1}{60}\zeta(3),\\
&\tau^{6,2}_6=-\frac{1}{720}\zeta(2),~\tau^{6,2}_7=0,~\tau^{6,2}_8=-\frac{1}{8!}.
\end{aligned}
\end{equation}
The equations
\begin{equation}
\Pi_{6,2}(1)=\sum_{i=0}^8 \tau^{6,2}_i \varpi^{6,2}_i(1)~\text{and}~\Pi_{6,2}(-1)=\sum_{i=0}^8 \tau^{6,2}_i \varpi^{6,2}_i(-1)
\end{equation}
give us the trivial identity $0=0$.

\subsection{\texorpdfstring{$\bm{\zeta(7,2)}$}{z(7,2)}}

When $k=7$, use the same method as in Section \ref{sec:sectionzeta21}, we obtain
\begin{equation}
\begin{aligned}
&\tau^{7,2}_0=8 \zeta(9)+6\zeta(2)\zeta(7)+\frac{31}{8}\zeta(3)\zeta(6)+7\zeta(4)\zeta(5),\tau^{7,2}_1=\frac{4501}{192}\zeta(8),\\
&\tau^{7,2}_2=3\zeta(7)+2\zeta(2)\zeta(5)+\frac{7}{4}\zeta(3)\zeta(4),\tau^{7,2}_3=\frac{65}{32}\zeta(6),\tau^{7,2}_4=\frac{1}{6} \zeta(5)+\frac{1}{12}\zeta(2)\zeta(3),\\
&\tau^{7,2}_5=\frac{19}{480} \zeta(4),\tau^{7,2}_6=\frac{1}{360}\zeta(3),\tau^{7,2}_6=\frac{1}{5040}\zeta(2),~\tau^{7,2}_8=0,\tau^{7,2}_9=\frac{1}{9!}.
\end{aligned}
\end{equation}
The equations
\begin{equation}
\Pi_{7,2}(1)=\sum_{i=0}^9 \tau^{7,2}_i \varpi^{7,2}_i(1)~\text{and}~\Pi_{7,2}(-1)=\sum_{i=0}^9 \tau^{7,2}_i \varpi^{7,2}_i(-1)
\end{equation}
give us the following lemma.
\begin{lemma}
\begin{equation}
\begin{aligned}
&\zeta(7,2)=\sum_{n=1}^\infty \frac{H_{n,2}}{(n+1)^7}=-\frac{37}{2}\zeta(9)+7\zeta(2)\zeta(7)+2\zeta(3)\zeta(6)+4\zeta(4)\zeta(5), \\
&\sum_{n=1}^\infty \frac{(-1)^nH_{n,2}}{(n+1)^7}=-\frac{5347}{512} \zeta(9)+\frac{321}{128}\zeta(2) \zeta(7)+\frac{31}{16}\zeta(3)\zeta(6)+\frac{7}{2} \zeta(4)\zeta(5).
\end{aligned}
\end{equation}
\end{lemma}

\subsection{\texorpdfstring{$\bm{\zeta(8,2)}$}{z(8,2)}}

When $k=8$, use the same method as in Section \ref{sec:sectionzeta21}, we obtain
\begin{equation}
\begin{aligned}
&\tau^{8,2}_0=-\frac{49363}{1280}\zeta(10),\tau^{8,2}_1=-8 \zeta(9)-6\zeta(2)\zeta(7)-\frac{31}{8}\zeta(3)\zeta(6)-7\zeta(4)\zeta(5),\\
&\tau^{8,2}_2=-\frac{4501}{384}\zeta(8),\tau^{8,2}_3=-\zeta(7)-\frac{2}{3}\zeta(2)\zeta(5)-\frac{7}{12}\zeta(3)\zeta(4),\tau^{8,2}_4=-\frac{65}{128}\zeta(6),\\
&\tau^{8,2}_5=-\frac{1}{30} \zeta(5)-\frac{1}{60}\zeta(2)\zeta(3),\tau^{8,2}_6=-\frac{19}{2880} \zeta(4),\tau^{8,2}_7=-\frac{1}{2520}\zeta(3), \\
&\tau^{8,2}_8=-\frac{1}{40320}\zeta(2),~\tau^{8,2}_9=0,\tau^{8,2}_{10}=-\frac{1}{10!}.
\end{aligned}
\end{equation}
The equations
\begin{equation}
\Pi_{8,2}(1)=\sum_{i=0}^{10} \tau^{8,2}_i \varpi^{8,2}_i(1)~\text{and}~\Pi_{8,2}(-1)=\sum_{i=0}^{10} \tau^{8,2}_i \varpi^{8,2}_i(-1)
\end{equation}
give us the trivial identity $0=0$.

\subsection{\texorpdfstring{$\bm{\zeta(9,2)}$}{z(9,2)}}

When $k=9$, use the same method as in Section \ref{sec:sectionzeta21}, we obtain
\begin{equation}
\begin{aligned}
&\tau^{9,2}_0=10\zeta(11)+8\zeta(2) \zeta(9)+\frac{127}{32}\zeta(3)\zeta(8)+\frac{21}{2}\zeta(4)\zeta(7)+\frac{31}{4} \zeta(5)\zeta(6),\tau^{9,2}_1=\frac{49363}{1280}\zeta(10),\\
&\tau^{9,2}_2=4 \zeta(9)+3\zeta(2)\zeta(7)+\frac{31}{16}\zeta(3)\zeta(6)+\frac{7}{2} \zeta(4)\zeta(5),\tau^{9,2}_3=\frac{4501}{1152}\zeta(8),\\
&\tau^{9,2}_4=\frac{1}{4}\zeta(7)+\frac{1}{6}\zeta(2)\zeta(5)+\frac{7}{48}\zeta(3)\zeta(4),\tau^{9,2}_5=\frac{13}{128}\zeta(6),\tau^{9,2}_6=\frac{1}{180} \zeta(5)+\frac{1}{360}\zeta(2)\zeta(3),\\
&\tau^{9,2}_7=\frac{19}{20160} \zeta(4),\tau^{9,2}_8=\frac{1}{20160}\zeta(3),\tau^{9,2}_9=\frac{1}{362880}\zeta(2),\tau^{9,2}_{10}=0,\tau^{9,2}_{11}=\frac{1}{11!}.
\end{aligned}
\end{equation}
The equations
\begin{equation}
\Pi_{9,2}(1)=\sum_{i=0}^{11} \tau^{9,2}_i \varpi^{9,2}_i(1)~\text{and}~\Pi_{9,2}(-1)=\sum_{i=0}^{11} \tau^{9,2}_i \varpi^{9,2}_i(-1)
\end{equation}
give us the following lemma.
\begin{lemma}
\begin{equation}
\begin{aligned}
&\zeta(9,2)=\sum_{n=1}^\infty \frac{H_{n,2}}{(n+1)^9}=-28\zeta(11)+9 \zeta(2)\zeta(9)+2\zeta(3)\zeta(8)+6\zeta(4)\zeta(7)+4\zeta(5)\zeta(6), \\
&\sum_{n=1}^\infty \frac{(-1)^nH_{n,2}}{(n+1)^9}=-\frac{18409}{1024}\zeta(11)+\frac{1793}{512}\zeta(2) \zeta(9)+\frac{127}{64}\zeta(3) \zeta(8)+\frac{21}{4}\zeta(4)\zeta(7)+\frac{31}{8} \zeta(5)\zeta(6).
\end{aligned}
\end{equation}
\end{lemma}

\subsection{Generalization}

Again for large $k$, it is practically very difficult to solve the linear equations given by 
\begin{equation} \label{eq:zeta22general}
\int_{-1/2}^{1/2} \phi^{-n} \Pi_{k,2} dt=\sum_{i=0}^{k+2} \tau^{k,2}_i \,\int_{-1/2}^{1/2} \varphi^{n} \varpi^{k,2}_i dt,~\tau^{k,2}_i \in \mathbb{C}.
\end{equation}
From our computations of $\tau^{k,2}_i$ for the cases where $k=2,3,4,5,6,7,8,9$, we have also observed that
\begin{equation}
\tau^{k,2}_i=-\frac{1}{i}\tau^{k-1,2}_{i-1},~i \geq 1.
\end{equation}
On the other hand, let $n=0$ in the formula \ref{eq:zeta22general}, and we have
\begin{equation} \label{eq:tauk2zeroequation}
\tau^{k,2}_0=- \sum_{i=1}^{k+2} \tau^{k,2}_i \int_{-1/2}^{1/2}  \varpi_i dt,
\end{equation}
where we have used the integrals
\begin{equation}
\int_{-1/2}^{1/2}  \Pi_{k,2} dt=0,~\int_{-1/2}^{1/2}   \varpi^{k,2}_0 dt=1.
\end{equation}
The other integrals in the formula \ref{eq:tauk2zeroequation} can also be evaluated easily
\begin{equation}
\begin{aligned}
\int_{-1/2}^{1/2}  \varpi^{k,2}_j dt&=\frac{\left(1+(-1)^j \right) (\pi i)^j}{2(1+j)},~j=1,\cdots,k;\\
\int_{-1/2}^{1/2}  \varpi^{k,2}_{k+1} dt&=\frac{\left(1+(-1)^{k+1} \right) (\pi i)^{k+1}}{2(k+2)}+(k+1)!\zeta(k+1);\\
\int_{-1/2}^{1/2}  \varpi^{k,2}_{k+2} dt&=\frac{\left(1+(-1)^{k+2} \right) (\pi i)^{k+2}}{2(k+3)}-(k+1)(k+2)!\zeta(k+2).
\end{aligned}
\end{equation}

\begin{conjecture} \label{conjecturetauk2}
The complex number $\tau^{k,2}_i,i \geq 1$ is always equal to $-\tau^{k-1,2}_{i-1}/i$. Together with formula \ref{eq:tauk2zeroequation}, it gives us a very efficient algorithm to compute $\tau^{k,2}_i$.
\end{conjecture}

\noindent In particular, we have the following corollary.
\begin{corollary}
The complex number $\tau^{k,2}_{k+2}$ is equal to $(-1)^{k+1}/(k+2)!$. When $k$ is an odd integer, the equations 
\begin{equation}
\Pi_{k,2}(1)=\sum_{i=0}^{k+2} \tau^{k,2}_i \varpi^{k,2}_i(1)~\text{and}~\Pi_{k,2}(-1)=\sum_{i=0}^{k+2} \tau^{k,2}_i \varpi^{k,2}_i(-1)
\end{equation}
will give us the values of
\begin{equation}
\zeta(k,2)=\sum_{n=1}^\infty \frac{H_{n,2}}{(n+1)^k}~\text{and}~\sum_{n=1}^\infty \frac{(-1)^nH_{n,2}}{(n+1)^k}
\end{equation}
in terms of the zeta values $\zeta(2),\cdots,\zeta(k+2)$. While when $k$ is even, the two equations become the trivial identity $0=0$.

\end{corollary}

\begin{remark}
The readers are referred to the paper \cite{Yang} for the similarities between the complex numbers $\tau^{k,1}_i$, $\tau^{k,2}_i$ and the periods of Calabi-Yau $n$-folds.
\end{remark}

\section{Further prospects} \label{sec:furtherprospects}

We will end this paper with several open questions.
\begin{enumerate}
\item Prove \textbf{Conjecture} \ref{conjecturetauk1} and \textbf{Conjecture} \ref{conjecturetauk2}.

\item Generalize the method to arbitrary double zeta values $\zeta(k,m)$.

\item Generalize the method to arbitrary MZVs.

\item Does there exist any connection between the results in this paper and the periods of Calabi-Yau manifolds \cite{Yang}.
\end{enumerate}


\appendix

\end{document}